\documentclass{article}

\usepackage{amsmath}
\usepackage{amsthm}
\usepackage{amssymb}
\usepackage{stmaryrd}
\usepackage{subfigure}
\usepackage{tikz}
\usetikzlibrary{calc}

\theoremstyle{plain}
\newtheorem{theorem}{Theorem}
\newtheorem{lemma}[theorem]{Lemma}
\newtheorem{corollary}[theorem]{Corollary}
\newtheorem{conjecture}[theorem]{Conjecture}
\newtheorem{proposition}[theorem]{Proposition}
\theoremstyle{definition}
\theoremstyle{remark}

\newcommand{\Mp}{\ensuremath{\mathrm{mp}}}
\newcommand{\Bc}{\ensuremath{\gamma_b}}
\newcommand{\D}{\text{diam}}
\newcommand{\R}{\text{rad}}
\renewcommand{\ell}{l}

\tikzstyle{vertex}=[circle, fill=black, inner sep= 0, minimum size = 4]
\tikzstyle{unselected}=[circle, draw, fill=white, inner sep= 0, minimum size = 4]
\tikzstyle{unknown}=[circle, fill=black, inner sep= 0, minimum size = 2]

\title{Broadcast domination and multipacking:\\bounds and the integrality gap}
\author{L.~Beaudou\footnote{LIMOS, Universit\'e Clermont Auvergne, Aubi\`ere (France). E-mails: laurent.beaudou@uca.fr, florent.foucaud@gmail.com} \and R.~C.~Brewster\footnote{Department of Mathematics and Statistics, Thompson Rivers University, Kamloops (Canada). E-mail: rbrewster@tru.ca} \and F.~Foucaud\footnotemark[1]}

\begin{document}

\maketitle


\begin{abstract}
  The dual concepts of coverings and packings are well studied in
  graph theory.  Coverings of graphs with balls of radius one and
  packings of vertices with pairwise distances at least 
  two are the well-known concepts of domination and independence, respectively.
  In 2001, Erwin introduced \emph{broadcast domination} in graphs, 
  a covering problem using balls of various radii
  where the cost of a ball is its radius. The minimum cost of a
  dominating broadcast in a graph $G$ is denoted by $\Bc(G)$. 
  The dual (in the sense of linear programming) 
  of broadcast domination is \emph{multipacking}: 
  a multipacking is a set $P \subseteq V(G)$ such that for any vertex $v$ 
  and any positive integer $r$, the ball of radius $r$ around $v$ contains at most $r$
  vertices of $P$. The maximum size of a multipacking in a graph $G$
  is denoted by $\Mp(G)$. Naturally, $\Mp(G) \leq \Bc(G)$.   
  Hartnell and Mynhardt proved that \mbox{$\Bc(G) \leq 3 \Mp(G) - 2$} 
  (whenever $\Mp(G)\geq 2$). In this paper, we show that
  \mbox{$\Bc(G) \leq 2\Mp(G) + 3$}.  Moreover, 
  we conjecture that this can be improved to $\Bc(G) \leq 2\Mp(G)$
  (which would be sharp). 
  \end{abstract}


\section*{Introduction}

The dual concepts of coverings and packings are well studied in
graph theory.  Coverings of graphs with balls of radius one and
packings of vertices with pairwise distances at least 
two are the well-known concepts of domination and independence respectively.
Typically we are interested in minimum (cost) coverings and maximum 
packings.  Natural questions to ask are for what graph do these dual
problems have equal (integer) values, and in the case they are not equal, can we
bound the difference between the two values? The second question is 
the focus of this paper.  

The particular covering problem we study is broadcast domination.
Let $G=(V,E)$ be a graph.
Define the \emph{ball of radius $r$ around $v$} by 
$N_r(v) = \{ u : d(u,v) \leq r \}$.  A \emph{dominating broadcast} of $G$
is a collection of balls $N_{r_1}(v_1), N_{r_2}(v_2), \dots, N_{r_t}(v_t)$
(each $r_i > 0$) such that $\bigcup_{i=1}^t N_{r_i}(v_i) = V$.
Alternatively, a dominating broadcast is a function $f: V \to \mathbb{N}$ 
such that for any vertex $u \in V$, there is a vertex $v \in V$ 
with $f(v)$ positive and $\mathrm{dist}(u,v) \leq f(v)$.   (The ball around
$v$ with radius $f(v)$ belongs to the covering.)
The \emph{cost} of a dominating broadcast $f$ is 
$\sum_{v \in V} f(v)$ and the minimum cost of a dominating broadcast in $G$, 
its \emph{broadcast number}, is denoted by $\Bc(G)$.\footnote{
  One may consider the cost to be any function of the powers
  (for example the sum of the squares), see e.g.~\cite{HeggernesLokshtanov2006}. 
  We shall stick to the classical convention of linear cost.
}

When broadcast
domination is formulated as an integer linear program, its dual
problem is \emph{multipacking}~\cite{Brewster2013,Teshima2012}. A
multipacking in a graph $G$ is a subset $P$ of its vertices such that
for any positive integer $r$ and any vertex $v$ in $V$, the ball of
radius $r$ centered at $v$ contains at most $r$ vertices of $P$. The
maximum size of a multipacking of $G$, its \emph{multipacking number},
is denoted by $\Mp(G)$.

Broadcast domination was introduced by Erwin~\cite{Erwin2001,
Erwin2004} in his doctoral thesis in 2001. Multipacking was then
defined in Teshima's Master's Thesis~\cite{Teshima2012} in 2012, see
also~\cite{Brewster2013}
(and~\cite{Brewster2017,hartnell_2014,Yang2015} for subsequent
studies). As we have already mentioned, this work fits into the 
general study of coverings and packings, which has a rich history 
in Graph Theory: Cornu\'ejols
wrote a monograph on the topic~\cite{Cornuejols2001}.

In early work, Meir and Moon~\cite{MeirMoon1975} studied 
various coverings and packings in trees, providing several inequalities relating 
the size of a minimum covering and a maximum packing.  Giving such inequalities 
connecting the parameters $\Bc$ and $\Mp$ is the focus of our work.
Since broadcast domination and multipacking are dual problems,
we know that for any graph $G$,
\begin{equation*}
  \Mp(G) \leq \Bc(G).
\end{equation*}

This bound is tight, in particular for strongly chordal graphs,
see~\cite{Farber84,Lubiw87,Teshima2012}.  
(In a recent companion work we prove equality
for grids~\cite{Beaudou2018}.)
A natural question comes to mind. How far
apart can these two parameters be? Hartnell and
Mynhardt~\cite{hartnell_2014} gave a family of graphs $(G_k)_{k \in
  \mathbb{N}}$ for which the difference between both parameters is
$k$. In other words, the difference can be arbitrarily
large. Nonetheless, they proved that for any graph $G$ with
$\Mp(G)\geq 2$,
\begin{equation*}
  \Bc(G) \leq 3 \Mp(G) - 2
\end{equation*} 
and asked~\cite[Section~5]{hartnell_2014} whether the factor~$3$ can
be improved. Answering their question in the affirmative, our main
result is the following.
\begin{theorem}
  \label{thm:bounding}
  Let $G$ be a graph. Then,
  \begin{equation*}
    \Bc(G) \leq 2 \Mp(G) + 3.
  \end{equation*}
\end{theorem}

Moreover, we conjecture that the additive constant in the bound of
Theorem~\ref{thm:bounding} can be removed.

\begin{conjecture}
  \label{conj:fac2}
  For any graph $G$, $\Bc(G) \leq 2 \Mp(G)$.
\end{conjecture}

In Section~\ref{sec:bound}, we prove Theorem~\ref{thm:bounding}. In
Section~\ref{sec:discussion}, we show that Conjecture~\ref{conj:fac2}
holds for all graphs with multipacking number at most~$4$. We conclude the paper with some discussions in Section~\ref{sec:remarks}.

\section{Proof of Theorem~\ref{thm:bounding}}
\label{sec:bound}

We want to bound the broadcast number of a graph by a function of its
multipacking number. We first state a key counting result which is
used throughout the remainder of this paper.

For any two relative integers $a$ and $b$ such that $a \leq b$,
$\llbracket a, b\rrbracket$ denotes the set $\mathbb{Z} \cap [a,b]$.

\begin{lemma}
  \label{lem:path}
  Let $G$ be a graph, $k$ be a positive integer and
  $(u_0,\ldots,u_{3k})$ be an isometric path in $G$. Let \mbox{$P=\{u_{3i} | i \in \llbracket
  0,k \rrbracket \}$} be the
  set of every third vertex on this path. Then, for any positive integer $r$ and any ball
  $B$ of radius $r$ in~$G$,
  \begin{equation*}
    |B \cap P| \leq \left\lceil \frac{2r+1}{3} \right\rceil.
  \end{equation*}
\end{lemma}

\begin{proof}
  Let $B$ be a ball of radius $r$ in $G$, then any two vertices in $B$
  are at distance at most $2r$. Since the path $(u_0,\ldots,u_{3k})$
  is isometric the intersection of the path and $B$ is included in a
  subpath of length $2r$. This subpath contains at most $2r+1$
  vertices and only one third of those vertices can be in $P$.
\end{proof}

 Any positive integer $r$ is greater than or equal to $\left\lceil
 \frac{2r+1}{3} \right\rceil$. Thus, Lemma~\ref{lem:path} ensures that
 $P$ is a valid multipacking of size $k+1$. We have the following  
 (see also~\cite{dun_al_2006}):
 
\begin{proposition}
For any graph $G$,
  \begin{equation*}
    \Mp(G) \geq \left\lceil\frac{\D(G)+1}{3}\right\rceil.
  \end{equation*}
\end{proposition}

Building on this idea, we have the following result.

\begin{theorem}
  \label{thm:main}
  Given a graph $G$ and two positive integers $k$ and $k'$ such that
  \mbox{$k' \leq k$}, if there are four vertices $x,y,u$ and $v$ in $G$ such
  that 
  \begin{equation*} d_G(x,u) = d_G(x,v) = 3k \text{, } d_G(u,v) = 6k \text{ and }d_G(x,y) = 3k
  + 3k',
  \end{equation*}
  then
  \begin{equation*}
    \Mp(G) \geq 2k + k'.
  \end{equation*}
\end{theorem}

\begin{proof}
  Let $(u_{-3k},\ldots,u_0,\ldots,u_{3k})$ be the vertices of an
  isometric path from $u$ to $v$ going through $x$. Note that $u =
  u_{-3k}$, $x = u_0$ and $v = u_{3k}$. We shall select every third
  vertex of this isometric path and let $P_1$ be the set $\{u_{3i}
  | i \in \llbracket -k, k \rrbracket\}$.

  We thus have already selected $2k+1$ vertices. In order to complete
  our goal, we need $k'-1$ additional vertices. Let $(x_0,\ldots,x_{3k + 3k'})$
  be the vertices of an isometric path from $x$ to $y$. Note that $x =
  x_0$ and $y = x_{3k+3k'}$. We shall select every third vertex on
  this isometric path starting at $x_{3k+6}$.  Formally, we let $P_2$
  be the set $\{ x_{3k+3(i+2)} | i \in \llbracket 0,k'-2
  \rrbracket\}$. Finally, we let $P$ be the union of $P_1$ and
  $P_2$. An illustration of this is displayed in
  Figure~\ref{fig:firstscheme}.
  
  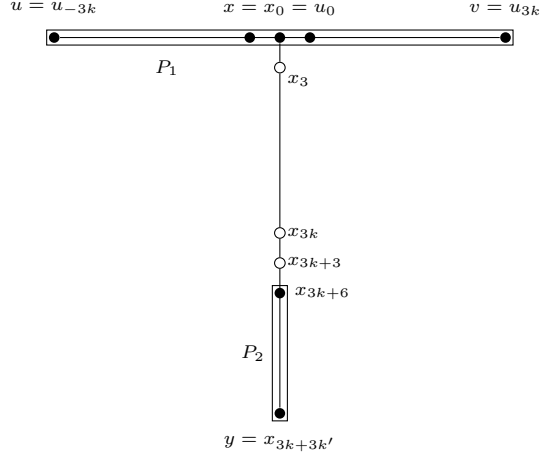
\begin{figure}[ht]
    \scriptsize
    \begin{center}
      \begin{tikzpicture}
        \node[vertex] (u) at (-3,0) {};
        \node[vertex] (v) at (3,0) {};
        \node[vertex] (x) at (0,0) {};
        \node[vertex] (y) at (0,-5) {};
        \node[vertex] (x3k6) at (0,-3.4) {};
        \node[unselected] (x3k) at (0,-2.6) {};
        \node[unselected] (x3k3) at (0,-3) {};
        \node[vertex] (um3) at (-.4,0) {};
        \node[vertex] (u3) at (.4,0) {};
        \node[unselected] (x3) at (0,-.4) {};
        \draw (u) -- (v);
        \draw (x) -- (x3) -- (x3k) -- (x3k3) -- (y);
        \node[above] at ($(x) + (0,.2)$) {$x = x_0 = u_0$};
        \node[above] at ($(u) + (0,.2)$) {$u = u_{-3k}$};
        \node[above] at ($(v) + (0,.2)$) {$v = u_{3k}$};
        \node[below] at ($(y) + (0,-.2)$) {$y = x_{3k+3k'}$};
        \node[below right] at ($(x3) + (0,0)$) {$x_3$};
        \node[right] at (x3k) {$x_{3k}$};
        \node[right] at ($(x3k6)+(.1,0)$) {$x_{3k+6}$};
        \node[right] at (x3k3) {$x_{3k+3}$};
        \node[left] at (-.1,-4.2) {$P_2$};
        \node[below] at (-1.5,-0.2) {$P_1$};
        \draw (-3.1,.1) rectangle (3.1,-.1);
        \draw (-.1,-3.3) rectangle (.1,-5.1);
      \end{tikzpicture}
    \end{center}
    \caption{Building of $P$.}
    \label{fig:firstscheme}
  \end{figure}

  Since every vertex of $P_2$ is at distance at least $3k + 6$ from
  $x$, while every vertex of $P_1$ is at distance at most $3k$ from
  $x$, we infer that $P_1$ and $P_2$ are disjoint. Thus $|P| =
  2k+k'$. We shall now prove that $P$ is a valid multipacking.

  Let $r$ be an integer between 1 and $|P| - 1$, and let $B$ be a ball
  of radius $r$ in $G$ (we do not care about the center of the
  ball). If this ball $B$ intersects only $P_1$ or only $P_2$, then we
  know by Lemma~\ref{lem:path} that it cannot contain more than $r$
  vertices of $P$. We may then consider that the ball $B$ intersects
  both $P_1$ and $P_2$. Let $l$ denote the greatest integer $i$ such
  that $x_{3k+3(i+2)}$ is in $B$ and in $P_2$. Let us name this vertex
  $z$. From this, we may say that
  \begin{equation}
    \label{eq:P2}
    |B \cap P_2| \leq l + 1
  \end{equation}

  Before ending this preamble, we state an easy
  inequality. For every integer $n$,
  \begin{equation}
    \label{eq:mod3}
    \left\lceil\frac{n}{3}\right\rceil \leq \frac{n}{3} + \frac{2}{3}
  \end{equation}

  We now split the remainder of the proof into two cases.

  \paragraph{Case 1: $3(l+2) \leq r$.} 
  In this case, we just use Lemma~\ref{lem:path} for
  $P_1$. We have 
  \begin{equation*}
    |B \cap P_1 | \leq \left\lceil \frac{2r + 1}{3} \right\rceil,
  \end{equation*}
  and by Inequality~\eqref{eq:mod3}, this quantity is bounded above by
  $\frac{2r+1}{3} + \frac{2}{3}$. We obtain with
  Inequality~\eqref{eq:P2},
  \begin{flalign*}
    &&|B \cap P| & \leq l+1 + \frac{2r+1}{3} + \frac{2}{3}&&\\
    && & \leq l+2 + \frac{2r}{3}&&\\
    && & \leq \frac{r}{3} + \frac{2r}{3} &\text{(by our case hypothesis)}&\\
    && & \leq r.&&
  \end{flalign*}
  Therefore, the ball $B$ contains at most $r$ vertices of $P$, as required.

  \paragraph{Case 2: $3(l+2) > r$.}
  Here we need some more insight. Recall that $l + 2 $
  cannot exceed $k'$ and that $k' \leq k$. Thus
  \begin{align*}
    r & < 3(l+2) \\
    & < 2k' + l +2\\
    & < 2k + l + 2,
  \end{align*}
  and since $r$ is an integer, we get 
  \begin{equation}
    \label{eq:nice}
    r \leq 2k + l + 1.
  \end{equation}
  
  We also note that any vertex $u_i$ for $|i| \leq 3k + 3(l+2) -
  (2r+1)$ is at distance at least $2r+1$ from $z$. By the triangle
  inequality $d(z,u_i) \geq d(z,x)-d(u_i,x)$, where $d(z,x)=3k + 3(l+2)$,
  and $d(u_i,x) = |i|$. Since the ball $B$ has
  radius $r$, no such vertex can be in $B$. Since we
  assumed that $B$ intersects $P_1$, not all the vertices of
  the $uv$-path are excluded from $B$. This means that 
  \begin{equation}
    \label{eq:nonzero}
    3k > 3k + 3(l+2) - (2r+1).
  \end{equation}
  We partition the vertices of $P_1$ into three sets: $U_L, U_M, U_R$.
  The vertex $u_i$ belongs to: $U_L$ if $i < -3k - 3(l+2) + 2(r+1)$;
  $U_M$ if $|i| \leq 3k + 3(l+2) - (2r+1)$; and $U_R$ if 
  $i > 3k + 3(l+2) - (2r+1)$.  See Figure~\ref{fig:case21}.
  The distance from $u = u_{-3k}$ to the first vertex (smallest positive index)
  in $U_R$ is then $6k + 3(l+2) - (2r+1) + 1$. We
  compare this distance with $2r+1$.
  
  \subparagraph{Case 2.1: $6k + 3(l+2) - (2r+1) + 1 \geq 2r+1$.}
  We match $U_L$ with $U_R$ so that each pair is at distance at
  least $2r+1$ (match $u_{-3k}$ with the first vertex in $U_R$ and so on, 
  as pictured in Figure~\ref{fig:case21}). Therefore the ball $B$ contains 
  at most one vertex of each matched pair. In other words, $B$ contains
  at most $\lceil |U_R|/3 \rceil$ vertices from $U_L \cup U_R$, and so
  \begin{equation*}
    |B \cap P_1| \leq \left\lceil \frac{3k - (3k + 3(l+2) -
      2r) + 1}{3} \right\rceil.
  \end{equation*}
  By using Inequality~\eqref{eq:P2} again,
  \begin{align*}
    |B \cap P| & \leq l+1 + \left\lceil \frac{2r+1}{3} \right\rceil - (l+2)\\
    & \leq r. 
  \end{align*}
  Therefore, the ball $B$ contains at most $r$ vertices of $P$, as required.

  \begin{figure}[ht]
    \begin{center}
      \scriptsize
      \subfigure[Case 2.1.]{
        \label{fig:case21}
        \begin{tikzpicture}[xscale=1.6]
          \node[vertex] (u) at (-3,0) {};
          \node[vertex] (v) at (3,0) {};
          \node[vertex] (x) at (0,0) {};
          \node[vertex] (um31) at (-2.4,0) {};
          \node[vertex] (u31) at (2.4,0) {};
          \draw (u) -- (v);
          \node[above] at ($(x) + (0,.2)$) {$u_0$};
          \node[left] at ($(u) + (-.2,0)$) {$u_{-3k}$};
          \node[right] at ($(v) + (.2,0)$) {$u_{3k}$};
          \node[below] at ($(-.5,-1.0) + (0,-.2)$) {$6k+3(l+2)-(2r+1)$};
          \node[below] at ($(-.95,-.5) + (0,-.2)$) {$2r+1$};
          \node[below] at ($(-2.65,0) + (0,-.2)$) {$U_L$};
          \node[below] at ($(2.65,0) + (0,-.2)$) {$U_R$};
          \node[below] at ($(0,0) + (0,-.2)$) {$U_M$};
          \draw (-2.1,.1) rectangle (2.1,-.1);
          \draw (-3.1,.1) rectangle (-2.2,-.1);
          \draw (2.2,.1) rectangle (3.1,-.1);
          \draw[dashed] (u) to[bend left] (u31);
          \draw[dashed] (um31) to[bend left] (v);
          \draw[arrows=<->] (-3,-.7) -- (1.1,-.7);
          \draw[arrows=<->] (-3,-1.2) -- (2.1,-1.2);
        \end{tikzpicture}
      } \quad 
      \subfigure[Case 2.2.]{
        \label{fig:case22}
        \begin{tikzpicture}[xscale=1.6]
          \node[vertex] (u) at (-3,0) {};
          \node[vertex] (v) at (3,0) {};
          \node[vertex] (x) at (0,0) {};
          \node[vertex] (um31) at (-2.4,0) {};
          \node[vertex] (u31) at (2.4,0) {};
          \draw (u) -- (v);
          \node[above] at ($(x) + (0,.2)$) {$u_0$};
          \node[left] at ($(u) + (-.2,0)$) {$u_{-3k}$};
          \node[right] at ($(v) + (.2,0)$) {$u_{3k}$};
          \node[below] at ($(-1,-.5) + (0,-.2)$) {$6k+3(l+2)-(2r+1)$};
          \node[below] at ($(-.4,-1) + (0,-.2)$) {$2r+1$};
          \node[below] at ($(-2.65,0) + (0,-.2)$) {$U'_L$};
          \node[below] at ($(2.65,0) + (0,-.2)$) {$U'_R$};
          \node[below] at ($(-1.6,0) + (0,-.2)$) {$U''_L$};
          \node[below] at ($(1.6,0) + (0,-.2)$) {$U''_R$};
          \node[below] at ($(0,0) + (0,-.2)$) {$U_M$};
          \draw (-1.1,.1) rectangle (1.1,-.1);
          \draw (-2.1,.1) rectangle (-1.2,-.1);
          \draw (1.2,.1) rectangle (2.1,-.1);
          \draw (-3.1,.1) rectangle (-2.2,-.1);
          \draw (2.2,.1) rectangle (3.1,-.1);
          \draw[dashed] (u) to[bend left] (u31);
          \draw[dashed] (um31) to[bend left] (v);
          \draw[arrows=<->] (-3,-.7) -- (1.1,-.7);
          \draw[arrows=<->] (-3,-1.2) -- (2.2,-1.2);
        \end{tikzpicture}
      }
    \end{center}
    \caption{Illustrations for Case 2.}
  \end{figure}
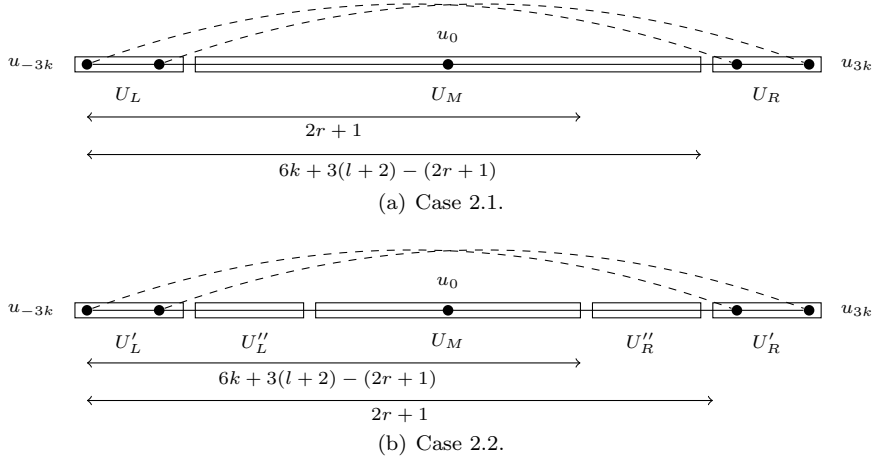
  
  \subparagraph{Case 2.2: $6k + 3(l+2) - (2r+1) + 1 < 2r+1$.}
  We partition each of $U_L$ and $U_R$ as shown
  in Figure~\ref{fig:case22}. The vertices that are
  distance at least $2r+1$ from a vertex in $U_L \cup U_R$ 
  are the sets $U'_L$ and $U'_R$,
  and those that are close to all other vertices are
  $U''_L$ and $U''_R$.
  We can match pairs of vertices $U'_L \cup U'_R$. 
  This allows us to say that the extremities of $P_1$ will
  contribute at most $\left\lceil \frac{6k - (2r+1) + 1}{3}
  \right\rceil$ which equals $2k + \lceil\frac{-2r}{3}\rceil$. Using
  again Inequality~\eqref{eq:mod3}, this is bounded above by $2k -
  \frac{2r}{3} + \frac{2}{3}$.

  For any integer $i$ between $6k + 3(l+2) - (2r+1) + 1$ and $2r$,
  vertices $u_{-i}$ and $u_{i}$ belong to $U''_L$ and $U''_R$ respectively.
  Such vertices may be in $B$. Since $P_1$ contains every third vertex on these two
  subpaths, this amounts to at most
  \begin{equation*}
    2 \left\lceil\frac{2r - 6k - 3(l+2) + (2r+1)}{3}\right\rceil
  \end{equation*}
  such vertices. This quantity is equal to 
  \begin{equation*}
    2\left\lceil \frac{4r+1}{3} \right\rceil -4k - 2(l+2),
  \end{equation*}
  which in turn, using Inequality~\eqref{eq:mod3} is bounded above by
  \begin{equation*}
    \frac{8r}{3} + 2 -4k -2(l+2).
  \end{equation*}
  
  By putting everything together, we derive that
  \begin{flalign*}
    && |B \cap P| & \leq (l+1) + \left(2k - \frac{2r}{3} + \frac{2}{3}\right) + \left(\frac{8r}{3} +2 -4k - 2(l+2)\right) &&\\
    &&& \leq 2r - 2k - l - \frac{1}{3}.&&
  \end{flalign*}
  But since $|B \cap P|$ is an integer, we may rewrite this last
  inequality as
  \begin{flalign*}&& |B \cap P| & \leq r + (r - 2k - l - 1) &&\\
    &&& \leq r. &\text{(by
    Inequality~\eqref{eq:nice})}&
  \end{flalign*}
  Thus, $|B \cap P|$ cannot exceed $r$ and the ball $B$ contains at most $r$ vertices of $P$, as required. This concludes the proof of Theorem~\ref{thm:main}.
\end{proof}

Theorem~\ref{thm:main} allows us to give a lower bound on the size of
a maximum multipacking in a graph in terms of its diameter and radius.

\begin{corollary}\label{coro:diam-rad}
  For any graph $G$ of diameter $d$ and radius r,
  \begin{equation*}
    \Mp(G) \geq \frac{d}{6} + \frac{r}{3} - \frac{3}{2}.
  \end{equation*}
\end{corollary}

\begin{proof}
  We just pick the integer $k$ such that $d$ can be expressed as
  $6k + \alpha$ where $\alpha$ is in $\llbracket 0,5 \rrbracket$ and
  the integer $k'$ such that $r$ can be expressed as $3k +
  3k'+\beta$ where $\beta$ is in $\llbracket 0,2\rrbracket$.

  We must have two vertices at distance $6k$ in $G$. On a shortest path of length $6k$, the middle
  vertex has some vertex at distance $3k+3k'$. We can then apply
  Theorem~\ref{thm:main}.
  \begin{align*}
    \Mp(G) &\geq 2k + k'\\
    & \geq \frac{1}{3}(d - \alpha) + \frac{1}{3} \left(r - \beta - \frac{1}{2}(d - \alpha)\right)\\
    & \geq \frac{d}{6} + \frac{r}{3} - \frac{9}{6}.\qedhere
  \end{align*}
\end{proof}

We can now finalize the proof of our main theorem.

\begin{proof}[Proof of Theorem~\ref{thm:bounding}]
Since the diameter of a graph is always greater than or equal to its
radius, we conclude from Corollary~\ref{coro:diam-rad} that
$$
\frac{\R(G)-3}{2} \leq \Mp(G) \leq \Bc(G) \leq \R(G).
$$
Hence, for any graph $G$,
\begin{equation*}
  \Bc(G) \leq 2 \Mp(G) + 3,
\end{equation*}
proving Theorem~\ref{thm:bounding}.
\end{proof}


Note that in our proof, we chose the length of the long path to be a
multiple of~$6$ for the reading to be smooth. We think that the same
ideas implemented with more care would work for multiples of~$3$.
This might slightly improve the additive constant in our bound, but we
believe that it would not be enough to prove
Conjecture~\ref{conj:fac2} (while adding too much complexity to the
proof).

\section{Proving Conjecture~\ref{conj:fac2} when $\Mp(G)\leq 4$}\label{sec:discussion}

The following collection of results shows that
Conjecture~\ref{conj:fac2} holds for graphs whose multipacking number
is at most~$4$.

\begin{lemma}\label{lemma-distances}
Let $G$ be a graph and $P$ a subset of vertices of $G$. If, for every
subset $U$ of at least two vertices of $P$, there exist two vertices
of $U$ that are at distance at least $2|U|-1$, then $P$ is a
multipacking of $G$.
\end{lemma}
\begin{proof}
  We prove the contrapositive. Let $G$ be a graph and $P$ a subset of
  its vertices which is not a multipacking. Then there is a ball $B$
  of radius $r$ which contains $r+1$ vertices of $P$. 

  Let $U$ be the set $B \cap P$, then $U$ has size at least
  $r+1$. Moreover, any two vertices in $U$ are at distance at most
  $2r$ which is stricly smaller than $2|U|-1$.
\end{proof}

\begin{proposition}\label{prop:mp=3}
Let $G$ be a graph. If $\Mp(G)=3$, then $\Bc(G)\leq 6$.
\end{proposition}
\begin{proof}
  We prove the contrapositive again.  Let $G$ be a graph with
  broadcast number at least 7. Then, the eccentricity of any vertex is
  at least 7 (otherwise we could cover the whole graph by broadcasting
  with power 6 from a single vertex).

  Let $x$ be any vertex of $G$. There must be a vertex $y$ at distance
  7 from $x$. Let $u$ be any vertex at distance 3 from $x$ and on a
  shortest path from $x$ to $y$. Then $u$ is at distance 4 from
  $y$. But $u$ has also eccentricity at least 7. So there is a vertex
  $v$ at distance 7 from $u$. By the triangle inequality, $v$ is at
  distance at least 4 from $x$ and at least 3 from $y$. Therefore the
  set $\{u,v,x,y\}$ satisfies the condition of
  Lemma~\ref{lemma-distances} and the multipacking number of $G$ is at
  least 4 (and so it is not equal to 3).
\end{proof}

The following proposition improves Theorem~\ref{thm:bounding} for 
graphs $G$ with $\Mp(G) \leq 6$ and shows that
Conjecture~\ref{conj:fac2} holds when $\Mp(G) = 4$.

\begin{proposition}\label{prop:mp=4}
Let $G$ be a graph. If $\Mp(G)\geq 4$, then $\Bc(G)\leq 3\Mp(G)-4$.
\end{proposition}
\begin{proof}
  For a contradiction, let $G$ be a counterexample, that is a graph
  with multipacking number $p$ at least 4 while $\Bc(G)\geq
  3p-3$. Then, the eccentricity of any vertex of $G$ is at least
  $3p-3$ (otherwise we could broadcast at distance $3p-4$ from a
  single vertex). Let $x$ be a vertex of $G$ and let $V_i$ denote the
  set of vertices at distance exactly~$i$ of $x$. By our previous
  remark, $V_{3p-3}$ is non-empty. Let $y$ be a vertex in $V_{3p-3}$
  and consider a shortest path $P_{xy}$ from $x$ to $y$ in $G$. Let
  $v_0=x$, and for $1\leq i\leq p-1$, let $v_i$ be the vertex on
  $P_{xy}$ belonging to $V_{3i}$ (thus $v_{p-1}=y$).

  Now, since $\Bc(G)\geq 3p-3$, there must be a vertex $u$ at distance
  at least $3p-3$ of $v_{p-2}$ (otherwise we could broadcast from that
  single vertex). Note that the triangle inequality ensures that the
  distance between $u$ and $v_i$ is at least $3+3i$ for $i$ between
  $0$ and $p-2$. The distance from $u$ to $v_{p-1}$ is at least $3p-6$
  which is at least 6 since $p$ is at least 4. Consider the set
  $P=\{u,v_0,\ldots, v_{p-1}\}$. We claim that $P$ is a multipacking
  of $G$ of size $p+1$, which is a contradiction.

  Let $B$ be a ball of radius $r$. Since $P_{xy}$ is an isometric
  path, Lemma~\ref{lem:path} ensures us that $B$ contains at most
  \begin{equation*}
    \left\lceil \frac{2r+1}{3} \right\rceil
  \end{equation*}
  vertices from $P \cap P_{xy}$ which is smaller than $r$. When $B$
  does not include $u$, the ball is satisfied. For balls that contain
  vertex $u$, the maximum size of $P \cap B$ is
  \begin{equation*}
    \left\lceil \frac{2r+1}{3} \right\rceil + 1.
  \end{equation*}
  Whenever $r$ is 4 or more, this quantity does not exceed $r$. So
  every ball with radius $4$ or more is satisfied. We still need to
  check balls of radius 1,2, and 3 which contain $u$.
  \begin{itemize}
  \item Balls of radius 1 are easy to check since every vertex of $P_{xy}$
    is at distance at least 3 from $u$.
  \item For balls of radius 2, it is enough to check that there is only one
    vertex at distance 4 or less from $u$ in $P \cap P_{xy}$.
  \item For balls of radius 3, there is only one way to select $u$ and three
  vertices in $P \cap P_{xy}$ within distance 6 from $u$. We should
  take $v_0, v_1$ and $v_{p-1}$. But since $v_0$ and $v_{p-1}$ are at
  distance $3p-3$ from each other, they cannot appear simultaneously in
  a ball of radius 3 (since $p$ is at least 4, $3p-3$ is at least 9).
  \end{itemize}
  
  Therefore $P$ is a multipacking of size $p+1$, which is a
  contradiction.
\end{proof}

\begin{corollary}
Let $G$ be a graph. If $\Mp(G)\leq 4$, then $\Bc(G)\leq 2\Mp(G)$.
\end{corollary}
\begin{proof}
When $\Mp(G)\leq 2$, this is shown in~\cite{hartnell_2014}. The case $\Mp(G)=3$ is implied by Proposition~\ref{prop:mp=3}, and the case $\Mp(G)=4$ follows from Proposition~\ref{prop:mp=4}.
\end{proof}

\section{Concluding remarks}\label{sec:remarks}

We conclude the paper with some remarks.

\subsection{The optimality of Conjecture~\ref{conj:fac2}}

We know a
few examples of connected graphs $G$ which achieve the conjectured
bound, that is, $\Bc(G)=2\Mp(G)$. For example, one can easily check
that $C_4$ and $C_5$ have multipacking number~$1$ and broadcast
number~$2$. In Figure~\ref{fig:twoFour}, we depict three examples
having multipacking number~$2$ and broadcast number~$4$. By making
disjoint unions of these graphs, we can build further extremal graphs
with arbitrary multipacking number. However, if we only consider
connected graphs, we do not even know an example with multipacking
number~$3$ and broadcast number~$6$. Hartnell and
Mynhardt~\cite{hartnell_2014} constructed an infinite family of
connected graphs $G$ with $\Bc(G)=\tfrac{4}{3}\Mp(G)$, but we do not
know any construction with a higher ratio. Are there arbitrarily large
connected graphs that reach the bound of Conjecture~\ref{conj:fac2}?

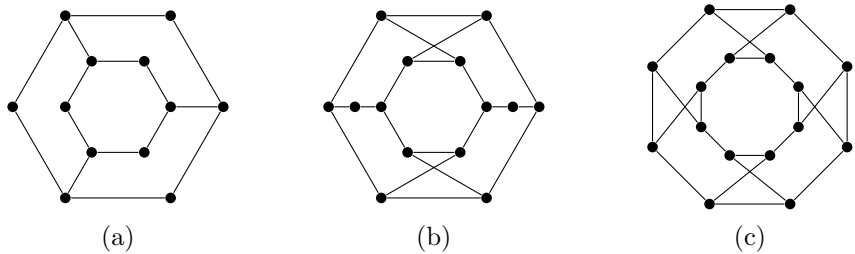
\begin{figure}[!ht]
\begin{center}
\scalebox{1.0}{\begin{tikzpicture}[join=bevel,inner sep=0.5mm,scale=0.7]
\node[vertex](0) at (0:1) {};
\node[vertex](1) at (60:1) {};
\node[vertex](2) at (120:1) {};
\node[vertex](3) at (180:1) {};
\node[vertex](4) at (240:1) {};
\node[vertex](5) at (300:1) {};
\node[vertex](a) at (0:2) {};
\node[vertex](b) at (60:2) {};
\node[vertex](c) at (120:2) {};
\node[vertex](d) at (180:2) {};
\node[vertex](e) at (240:2) {};
\node[vertex](f) at (300:2) {};

\draw[-] (0)--(1)--(2)--(3)--(4)--(5)--(0)
         (a)--(b)--(c)--(d)--(e)--(f)--(a)
         (0)--(a) (2)--(c) (4)--(e);

\node at (270:2.5) {(a)};

\begin{scope}[xshift=6cm]
\node[vertex](0) at (0:1) {};
\node[vertex](1) at (60:1) {};
\node[vertex](2) at (120:1) {};
\node[vertex](3) at (180:1) {};
\node[vertex](4) at (240:1) {};
\node[vertex](5) at (300:1) {};
\node[vertex](a) at (0:2) {};
\node[vertex](b) at (60:2) {};
\node[vertex](c) at (120:2) {};
\node[vertex](d) at (180:2) {};
\node[vertex](e) at (240:2) {};
\node[vertex](f) at (300:2) {};
\node[vertex](x) at (0:1.5) {};
\node[vertex](y) at (180:1.5) {};

\draw[-] (0)--(1)--(2)--(3)--(4)--(5)--(0)
         (a)--(b)--(c)--(d)--(e)--(f)--(a)
         (1)--(c) (2)--(b) (4)--(f) (5)--(e) (0)--(x)--(a) (3)--(y)--(d);         

\node at (270:2.5) {(b)};
\end{scope}

\begin{scope}[xshift=12cm,rotate=-22.5]
  \foreach \i in {0,1,2,3,4,5,6,7}
  {
    \node[vertex](x\i) at (45*\i:1) {};
    \node[vertex](y\i) at (45*\i:2) {};    
  }  
\draw[-] (x0)--(x1)--(x2)--(x3)--(x4)--(x5)--(x6)--(x7)--(x0)
         (y0)--(y1)--(y2)--(y3)--(y4)--(y5)--(y6)--(y7)--(y0)
         (x0)--(y1)  (y0)--(x1)  (x2)--(y3)  (y2)--(x3)
         (x4)--(y5)  (y4)--(x5)  (x6)--(y7)  (x7)--(y6);  
         
\node at (292.6:2.5) {(c)};         
\end{scope}
\end{tikzpicture}}
\end{center}
\caption{\label{fig:twoFour} Graphs with multipacking number $2$ and broadcast number $4$.
Graph (b) comes from L.~Teshima's Master Thesis~\cite{Teshima2012} and (c) was found by C.~R.~Dougherty (private communication).}
\end{figure}

\subsection{An approximation algorithm}

The computational complexity of broadcast domination has been
extensively studied, see for
example~\cite{Dabney2009,HeggernesLokshtanov2006} and references
of~\cite{Brewster2013,Teshima2012,Yang2015}. It is particularly
interesting to note that, unlike most other natural covering problems,
broadcast domination is solvable in polynomial (sextic)
time~\cite{HeggernesLokshtanov2006}. It is not known whether this is
also the case for multipacking, but a cubic-time algorithm exists for
strongly chordal graphs~\cite{Brewster2017,Yang2015}, as well as a
linear-time algorithm for
trees~\cite{Brewster2013,Brewster2017,Yang2015}. We note that our
proof of Theorem~\ref{thm:bounding}, being constructive, implies the
existence of a $(2+o(1))$-factor approximation algorithm for the
multipacking problem.

\begin{corollary}
There is a polynomial-time algorithm that, given a graph $G$,
constructs a multipacking of $G$ of size at least
$\frac{\Mp(G)-3}{2}$.
\end{corollary}
\begin{proof}
  To construct the multipacking, one first needs to compute the radius
  $r$ and diameter $d$ of the graph $G$. Then, as described in the
  proof of Corollary~\ref{coro:diam-rad}, we compute $\alpha$ and $k$,
  and find the four vertices $x$, $y$, $u$, $v$ and the two isometric
  paths $P_1$ and $P_2$ described in Theorem~\ref{thm:main}. Finally,
  we proceed as in the proof of Theorem~\ref{thm:main}, that is, we
  essentially select every third vertex of these two paths to obtain
  the multipacking $P$. All distances and paths can be computed in
  polynomial time using classic methods. By
  Corollary~\ref{coro:diam-rad}, $P$ has size at least
  $\frac{\R(G)-3}{2}$. Since $\Mp(G)\leq \R(G)$, the
  approximation factor follows.
\end{proof}


\end{document}